\documentclass[12pt]{article}
\usepackage{amsmath,amsfonts,amsthm,amscd,upref,amstext}
\usepackage[dvips]{graphicx}
\usepackage{subfigure}
\usepackage{pb-diagram,pb-xy}
\usepackage[all]{xy}
\usepackage[usenames, dvipsnames]{color}

\newtheorem{prop}{Proposition}[section]
\newtheorem{thm}[prop]{Theorem}
\newtheorem{defi}[prop]{Definition}

\newtheorem{cor}[prop]{Corollary}

\newtheorem{rem}[prop]{Remark}
\newtheorem{lem}[prop]{Lemma}

\newcommand{\N}{\mathbb{N}}
\newcommand{\Z}{\mathbb{Z}}

\numberwithin{equation}{section}

\begin{document}

\title{Graded Version of Some Basic Theorems on Local Cohomology to a Pair of Ideals}
\author{P. H. Lima\thanks{Work partially supported by CAPES-Brazil 10056/12-2.}\,\,\,\,and\,\,\,V. H. Jorge P\'erez
\thanks{Work partially supported by CNPq-Brazil - Grants 309316/2011-1,
and FAPESP Grant 2012/20304-1. 2000 Mathematics Subject Classification 13D45, 13A02, 13E10. {\it Key words}: Coefficient ideals, integral closure, reduction, multiplicity}}

\date{}
\maketitle

\begin{abstract}
In this paper, we prove some well-known results on local cohomology with respect to a pair of ideals in graded version, such as, Independence Theorem, Lichtenbaum-Harshorne Vanishing Theorem, Basic Finiteness and Vanishing Theorem, among others. Besides, we present a generalized version of Melkersson Theorem about Artinianess of modules and a result concerning Artinianess of local cohomology modules.
\end{abstract}

\section{Introduction}
Local cohomology with respect to a pair of ideals was firstly defined in \cite{TYY}, where the authors generalized the usual notion of local cohomology module and studied its various properties such as the relation between the usual local cohomology module, $H_{I}^{i}(M)$, and the one defined to a pair of ideals, $H_{I,J}^{i}(M)$, vanishing and nonvanishing theorems, the Generalized Version of Lichtenbaum-Hartshorne Theorem, among others.

Having the above results as motivation,
the aim of this paper is to present in graded version some basic theorems on cohomology with respect to a pair of ideals, such as, Independence Theorem, Lichtenbaum-Harshorne Vanishing Theorem, Basic Finiteness and Vanishing Theorems, and assertions concerning Artinianess and depth with respect to a pair of ideals.

The organization of this paper is as follows.

Let $R$ be a graded ring and let $I \subseteq R$ be
a graded ideal, $J$ an arbitrary ideal and $M$ a graded $R$-module. Let $R_+$ denote the irrelevant ideal of $R$, that is, the ideal generalized by elements of positive degree.

In section 2, we view the local cohomology $^{*}H^{i}_{I,J}(M)$ as a graded module and express it in terms of usual local cohomology modules. To be more precise, we denote by $^{*}\widetilde{W}(I,J)$ the set of homogeneous ideals $\mathfrak{c}$ of $R$ such that
$I^{n} \subseteq \mathfrak{c}+J$ for some integer $n$, and then, show that
$$
^{*}H_{I,J}^{i}(M)\cong {\underrightarrow{\lim}_{\mathfrak{c}\in ^{*}\widetilde{W}(I,J)} \hspace{0.01mm}} ^{*}H^{i}_{\mathfrak{c}}(M).
$$
We also present the graded version of the Independence Theorem and \\ Lichtenbaum-Harshorne Vanishing theorem for a pair of ideals.

In section 3, we suppose $R$ is a positively graded Noetherian ring which is standard, that is, $R=R_0[R_1]$, where $R_0$ is local ring.
In \cite[Theorem 2.1]{RS}, the authors show that if
$n={\rm sup} \{i \ :  H^{i}_{R_{+}}(M)\neq 0  \}$ then the $R$-module
$
H_{R_{+}}^{n}(M)/\mathfrak{m}_{0}H_{R_{+}}^{n}(M)
$ is Artinian. We generalize this result for the case of local cohomology with respect to a pair of homogeneous ideals, besides showing that
$
H_{R_{+}, J}^{i}(M)/(\mathfrak{m}_{0}R +J)H_{R_{+}, J}^{i}(M)
$ is actually Artinian for all $i\geq 0$.
Furthermore, we prove that
$$
\dim M/(\mathfrak{m}_{0}R+J)M = {\rm sup} \{i \ :  H^{i}_{R_{+},J}(M)\neq 0  \},
$$
a generalization for \cite[Lemma 3.4]{BH}.
We also present a new version, with respect to a pair of ideals, for Melkersson's Theorem about Artinianess.

In section 4, the module $M$ and the ring $R$ are assumed to be a Cohen-Macaulay, and then, it is obtained an expression to the number
$$
{\rm inf}\{i\in\N_0| H_{R_+,J}^i(M)\neq 0\}.
$$

In section 5, $R_0$ is assumed to be a local ring with infinite residual field.
It is well-known that, for all $i\geq 0$, $H_{R_+}^i(M)_{n}$ is finitely generated $R_{0}$-module for $n \in \mathbb{Z}$ and $H_{R_+}^i(M)_{n}=0,$ for $n$ sufficiently large.  We give a positive answer for this in the case of cohomology modules with respect to a pair of ideals. If $J$ is generated by elements of zero degree then, for $i \geq 1$, $H_{R_+,J}^i(M)_{n}=0,$ for $n$ sufficiently large and $H_{R_+,J}^i(M)_{n}$ is a finitely generated $R_{0}$-module for all $n\in \mathbb{Z}$. Finally, we prove a result about asymptotical stable (Theorem \ref{ass_stable}).

\section{Graded versions for a pair of ideals}

In this section, we introduce a grading to $H^{i}_{I,J}(M)$, making this a graded module. The result \cite[Theorem 3.2]{TYY}, Independence theorem for a pair of ideals and Generalized Version of Lichtenbaum-Hartshorne Theorem are again presented now from the point of view of graded modules.

(A) Let $R$ be a graded ring and let $I \subseteq R$ be
a graded ideal and $J$ an arbitrary ideal. If $M$ is a graded $R$-module then $\Gamma_{I,J}(M)$ is a graded
submodule of $M$. In fact, pick $m=(m_{k})\in \Gamma_{I,J}(M)$, so $I^{n}\subseteq (0:m)+J$.
It is easy to see that $I^{n}\subseteq (0:m_{k})+J$, for all $i$ and each $k$. Then define
$\Gamma_{I,J}(M)_{i}=\{m\in M_{i}: mI^{n} \subseteq mJ \text{ for some positive integer } n\geq 1 \}$.

(B) For a homomorphism $f:M \rightarrow N$, we have $f(\Gamma_{I,J}(M)) \subseteq \Gamma_{I,J}(N)$,
so that there is a mapping $\Gamma_{I,J}(f):\Gamma_{I,J}(M) \rightarrow \Gamma_{I,J}(N)$,
which is the restriction of $f$ to $\Gamma_{I,J}(M)$. Thus $\Gamma_{I,J}$ is an additive functor on the category
of all graded $R$-modules.

(C) Since the category of the graded modules has enough injectives, we can form the $i$-th right derived functor of $\Gamma_{I,J}$ (on the category of the graded modules), which will be denoted by $^{*}H^{i}_{I,J}$ ($i\geq 0$). For a graded $R$-module $M$, we shall refer to $^{*}H^{i}_{I,J}(M)$ to be the $i$-th graded local cohomology module
of $M$ with respect to the pair of ideals $(I,J)$.

(D) One can derive, by functor properties, that given an exact sequence $0 \rightarrow M \rightarrow N \rightarrow P \rightarrow 0$ of graded $R$-modules, there is a long exact sequence
$$\begin{array}{ccccccccc}
    0 & \rightarrow & ^{*}H^{0}_{I,J}(M) & \rightarrow & ^{*}H^{0}_{I,J}(N) & \rightarrow & ^{*}H^{0}_{I,J}(P) & \rightarrow &\\
     & \rightarrow & ^{*}H^{1}_{I,J}(M) & \rightarrow & ^{*}H^{1}_{I,J}(N) & \rightarrow & ^{*}H^{1}_{I,J}(P) & \rightarrow & \cdot \cdot \cdot ,\\
\end{array}
$$
of graded modules
with respect to a pair of ideals.

\begin{defi}
We denote by $^{*}\widetilde{W}(I,J)$ the set of homogeneous ideals $\mathfrak{c}$ of $R$ such that
$I^{n} \subseteq \mathfrak{c}+J$ for some integer $n$. We also define a partial order for this set:
$$
\mathfrak{c} \leq \mathfrak{d} \text{ if } \mathfrak{c} \supseteq \mathfrak{d}, \text{ for } \mathfrak{c},\mathfrak{d}\in \hspace{0.01mm} ^{*}\widetilde{W}(I,J).
$$
If $\mathfrak{a} \leq \mathfrak{b}$ we obtain the inclusion map $\Gamma_{\mathfrak{a}}(M) \hookrightarrow \Gamma_{\mathfrak{b}}(M)$. The order relation on $ ^{*}\widetilde{W}(I,J)$ and the inclusion maps turn
$\{\Gamma_{\mathfrak{a}}(M) \}_{\mathfrak{a}\in ^{*}\widetilde{W}(I,J)}$
into a direct system of graded $R$-modules.
\end{defi}

\begin{prop}\label{lim}
Let $R$ be a graded ring, $I$ a graded ideal, $J$ an arbitrary ideal of $R$ and $M$ a graded $R$-module. Then there is a natural graded isomorphism
$$
^{*}H_{I,J}^{i}(M)\cong {\underrightarrow{\lim}_{\mathfrak{c}\in ^{*}\widetilde{W}(I,J)} \hspace{0.01mm}} ^{*}H^{i}_{\mathfrak{c}}(M)
$$
\end{prop}

\begin{proof}
Firstly we observe that
$\Gamma_{I,J}(M)=\bigcup_{\mathfrak{c} \in ^{*}\widetilde{W}(I,J)} \Gamma_{\mathfrak{c}}(M)$.
In fact, if $x \in \bigcup_{\mathfrak{c} \in ^{*}\widetilde{W}(I,J)} \Gamma_{\mathfrak{c}}(M)$, we have
$I^{m}\subseteq \mathfrak{c} + J$ and $xI^{n}=0$ for positive integers $n,m$ and some $\mathfrak{c}\in ^{*}\widetilde{W}(I,J)$. Since $I^{mn}\subseteq (\mathfrak{c}+J)^{n}\subseteq \mathfrak{c}^{n}+J$,
we have $I^{mn}x \subseteq Jx$, that is, $x\in \Gamma_{I,J}(M)$.
Now let $x \in \Gamma_{I,J}(M)$ homogeneous. Then $I^{n} \subseteq \mathfrak{c}+J$, where $\mathfrak{c}=\text{ann}(x)$, so $x\mathfrak{c}=0$ and $x\in \Gamma_{\mathfrak{c}}(M)$. If $x$ is not homogeneous, we write
$x=x_{1}+...+x_{r}$, where $x_{i}$ is homogeneous. We have then $\mathfrak{a}_{i}=\text{ann}(x_{i})$ for each $i$. Hence
$x\mathfrak{a}_{1}\cdot \cdot \cdot \mathfrak{a}_{r}=0$. Therefore $x\in \Gamma_{\mathfrak{c}}(M)$, where
$\mathfrak{c}=\mathfrak{a}_{1}\cdot \cdot \cdot \mathfrak{a}_{r}$.

Let $0 \rightarrow M \rightarrow N \rightarrow P \rightarrow 0$ be an exact sequence of $R$-modules. For each $\mathfrak{c}\in ^{*}\widetilde{W}(I,J)$, we have a long exact sequence
$$\begin{array}{ccccccccc}
    0 & \rightarrow & ^{*}H^{0}_{\mathfrak{c}}(M) & \rightarrow & ^{*}H^{0}_{\mathfrak{c}}(N) & \rightarrow & ^{*}H^{0}_{\mathfrak{c}}(P) & \rightarrow &\\
     & \rightarrow & ^{*}H^{1}_{\mathfrak{c}}(M) & \rightarrow & ^{*}H^{1}_{\mathfrak{c}}(N) & \rightarrow & ^{*}H^{1}_{\mathfrak{c}}(P) & \rightarrow & \cdot \cdot \cdot ,\\
\end{array}
$$
which we can take the limit to obtain a long exact sequence
$$
{\small
\begin{array}{llllllll}
    0 & \rightarrow & \underrightarrow{\lim}_{\mathfrak{c}\in ^{*}\widetilde{W}(I,J)} \ ^{*}H^{0}_{\mathfrak{c}}(M) & \rightarrow & \underrightarrow{\lim}_{\mathfrak{c}\in ^{*}\widetilde{W}(I,J)} \ ^{*}H^{0}_{\mathfrak{c}}(N) & \rightarrow & \underrightarrow{\lim}_{\mathfrak{c}\in ^{*}\widetilde{W}(I,J)} \ ^{*}H^{0}_{\mathfrak{c}}(P) & \\
     &  \rightarrow & \underrightarrow{\lim}_{\mathfrak{c}\in ^{*}\widetilde{W}(I,J)} \  ^{*}H^{1}_{\mathfrak{c}}(M) & \rightarrow & \underrightarrow{\lim}_{\mathfrak{c}\in  ^{*}\widetilde{W}(I,J)} \ ^{*}H^{1}_{\mathfrak{c}}(N) & \rightarrow & \cdots  .&  \\
\end{array}
}
$$

Since
$^{*}H^{i}_{\mathfrak{c}}(E)=0$ for any $^{*}$injective $R$-module $E$ and any positive integer $i$, $\underrightarrow{\lim}_{\mathfrak{c}\in ^{*}\widetilde{W}(I,J)}\ ^{*}H^{i}_{\mathfrak{c}}(E)=0$.
Therefore, one may conclude that
$$\{ \underrightarrow{\lim}_{\mathfrak{c}\in ^{*}\widetilde{W}(I,J)} \hspace{0.01mm} ^{*}H^{i}_{\mathfrak{c}} \ \ \ | \ i=0,1,2,...\}$$
\noindent is a system of right derived functors of $^{*}\Gamma_{I,J}$ (see \cite[Theorem 12.3.1]{BS}), and that there is the desired graded isomorphism.
\end{proof}

\begin{rem}\label{Remark2.3}
Consider the above setup. Since
$\Gamma_{I,J}(M)=\bigcup_{\mathfrak{c} \in ^{*}\widetilde{W}(I,J)} \Gamma_{\mathfrak{c}}(M)$, one may conclude similarly to the above proof that
$$
H_{I,J}^{i}(M)\cong \underrightarrow{\lim}_{\mathfrak{c}\in ^{*}\widetilde{W}(I,J)} \hspace{0.01mm} H^{i}_{\mathfrak{c}}(M).
$$
\end{rem}

The following result is obtained from the previous proposition once $^{*}H_{\mathfrak{c}}^{i}(E)=0$ for any graded ideal $\mathfrak{c}$ and an $^{*}$injective graded $R$-module $E$.

\begin{prop}\label{acyclic}
Let $R$ be a graded ring, I a graded ideal and $J$ an arbitrary ideal. Let $E$ be an $^{*}$injective graded $R$-module. Then
$H_{I,J}^{i}(E)=0$.
\end{prop}

By Proposition \ref{acyclic}, we have the following

\begin{prop}
Let $R$ be a graded ring, $I$ a graded ideal and $J$ an arbitrary ideal. Let $M$ be an $R$-graded module. There is an isomorphism
$$
^{*}H_{I,J}^{i}(M) \cong H_{I,J}^{i}(M),
$$
for all $i$ as underlying $R$-modules.
\end{prop}

\begin{thm}\emph{(Graded independence theorem for a pair of ideals)}

Assume $R=\bigoplus_{n\in \mathbb{Z}} R_{n}$ is a graded ring, $I \subset R$ a graded ideal and $J \subset R$ an arbitrary ideal. Let $R'=\bigoplus_{n\in \mathbb{Z}} R_{n}'$ be another Noetherian
graded ring, $f:R \rightarrow R'$ a graded homomorphism of rings such that
$f(I)=JR'$ and $M'$ an $R'$-module.
\begin{enumerate}
  \item For each $i\in \mathbb{N}_{0}$, both the cohomology modules $$H_{IR',JR'}^{i}(M') \ \text{  and  } \ H_{I,J}^{i}(M')$$ are graded $R$-modules;
  \item For each $i\in \mathbb{N}_{0}$, there exists a graded isomorphism $$H_{IR',JR'}^{i}(M') \cong H_{I,J}^{i}(M')$$
  of graded $R$-modules.
\end{enumerate}
\end{thm}

\begin{proof}
The first item follows because $f$ is homogeneous.

The assumption $f(I)=JR'$ gives $\Gamma_{IR',JR'}^{i}(M') = \Gamma_{I,J}^{i}(M')$.
By \cite[Theorem 2.7]{TYY}, for each graded $R$-module $M'$, one has an isomorphism
$$
H_{IR',JR'}^{i}(M') \cong H_{I,J}^{i}(M').
$$
By using the grading on $H_{I,J}^{i}(M')$ obtained from first item, we can turn this isomorphism into a graded isomorphism.
On the other hand, given an $R$-graded homomorphism $f:M'\rightarrow N'$, we have a natural commutative diagram
$$
\begin{array}{ccc}
  H_{I,J}^{i}(M')  & \longrightarrow & H_{I,J}^{i}(N') \\
  \downarrow &   & \downarrow \\
  H_{IR',JR'}^{i}(M') & \longrightarrow & H_{IR',JR'}^{i}(N').
\end{array}
$$
Hence, as $H_{IR',JR'}^{i}(E)$ equals zero (by Proposition \ref{acyclic}), these (new) gradings coincide with the ones from item $\emph{1}$. So, one can conclude the second item.
\end{proof}

\begin{lem}\label{completion}
Let $R$ be a graded local ring of unique graded maximal ideal $\mathfrak{m}$. Let $d$ be the dimension of $R$, $I$ a graded ideal and $J$ an arbitrary ideal of $R$. Then
$H^{d}_{I,J}(R)  \cong H^{d}_{I\widehat{R},J\widehat{R}}(\widehat{R})$.
\end{lem}

\begin{proof}
By \cite[Theorem 2.1]{CW} (the same proof is true for the non-local case), $H^{d}_{I,J}(R)$ is Artinian.
So, by \cite[Lemma 4.8]{TYY}, we have
$$
H^{d}_{I,J}(R)\cong H^{d}_{I,J}(R)\otimes_{R} \widehat{R} \cong H^{d}_{I,J}(\widehat{R}).
$$
On the other hand, by \cite[Theorem 2.7]{TYY}, one has $H^{d}_{I,J}(\widehat{R})\cong H^{d}_{I\widehat{R},J\widehat{R}}(\widehat{R})$. The lemma is then concluded.
\end{proof}

\begin{thm} \emph{(Graded Lichtenbaum-Hartshorne Vanishing Theorem to a pair of ideals)}

Let $R=\oplus_{n\geq 0}R_{n}$ be a positively graded ring and an integral domain of dimension $d$. Assume $R_{0}$ is a complete local ring. Let $I$ be a proper graded ideal and $J$ an arbitrary ideal of $R$. Then $H^{d}_{I,J}(R)=0$.
\end{thm}

\begin{proof}If $J=0$, it reduces to the usual case, already proved. Suppose then $J$ is a nonzero ideal.
Let $\mathfrak{m}$ be the unique graded maximal ideal of $R$. We know the completion $\widehat{R_{\mathfrak{m}}}$ is isomorphic to the $\mathfrak{m}$-adic completion $\widehat{R}$ of $R$. Further, since $R_{0}$ is complete, it is known that $\widehat{R}$ is a domain.
By the generalized Lichtenbaum-Harshorne Vanishing theorem (\cite[Theorem 4.9]{TYY}) one can deduce
$H^{d}_{IR_{\mathfrak{m}},JR_{\mathfrak{m}}}(R_{\mathfrak{m}})=0$. On the other hand, by Lemma \ref{completion}, one obtains
$$
H^{d}_{IR_{\mathfrak{m}},JR_{\mathfrak{m}}}(R_{\mathfrak{m}})\cong  H^{d}_{I,J}(R).
$$
Therefore $H^{d}_{I,J}(R)=0$.
\end{proof}

\begin{prop}\emph{(\cite[Corollary 4.2]{TYY})}\label{torsion}
Let $R$ be a graded local ring of unique graded maximal ideal $\mathfrak{m}$
and $M$ be a finite graded module over $R$. Let $I,J$ be graded ideals of $R$.
If $H_{I,J}^{i}(M)=0$ for all integers $i > 0$, then $M$ is an $(I,J)$-torsion $R$-module.
\end{prop}

\begin{proof}
Set $N=M/\Gamma_{I,J}(M)$. We need to prove that $N=0$. Suppose $N\neq 0$. By \cite[Corollary 1.13]{TYY}, we have $\Gamma_{I,J}(N)=0$ and $H_{I,J}(N)=H_{I,J}(M)=0$ for all $i>0$.
Since by hypothesis $I,J$ are graded ideals, we have $\mathfrak{m} \in W(I,J)$, so that
$$
\inf \{ \text{depth} \hspace{0.05cm} N_{\mathfrak{p}} \ | \ \mathfrak{p} \in W(I,J) \} \leq \text{depth} \hspace{0.05cm} N_{\mathfrak{m}}= \text{depth} \hspace{0.05cm} N< \infty.
$$
By using \cite[Theorem 4.1]{TYY}, we obtain $H_{I,J}(N)\neq 0$ for some integer $i \leq \text{depth} N$, which is a contradiction.
\end{proof}

\section{Top local cohomology}
In this section, we give a version for Melkersson's Theorem concerning Artinianess of module with respect to a pair of ideals (Proposition \ref{prop_melkerson}). Also it is obtained a result about Artinianess of local cohomology with respect to a pair of ideal (Theorem \ref{artianianess}). To conclude the section, we find the top of the local cohomology with respect to a pair of ideals (Theorem \ref{Theorem5}).

Throughout this section, let $R= \oplus_{d\geq 0}R_{d}$ denote a positively graded commutative Noetherian ring, which is standard, that is, $R=R_{0}[R_{1}]$.
Assume $R_{0}$ is a local ring of maximal ideal $\mathfrak{m}_{0}$. Set $R_{+}=\oplus_{i>0} R_{i}$, the irrelevant ideal of $R$. Let $M=\oplus_{d\in \mathbb{Z}}M_{d}$ be a finite graded $R$-module. Let $M[a]$ denote the graded module $a$-shift of $M$, defined by $M[a]_{i}=M_{i+a}$.

\begin{rem}
It can be derived from Proposition \ref{lim} that $H_{I,J}^{i}(M[a])\cong  H_{I,J}^{i}(M)[a]$ as homogeneous modules.
\end{rem}

\begin{prop}\emph{(\cite[Theorem 1.3]{M})}\label{prop_melkerson}
Let $M$ be an $(I,J)$-torsion $R$-module for which $(JM:_{M}I)$ is an Artinian module. Then
$M$ is an Artinian module.
\end{prop}

\begin{proof}
If $M$ is an $(I,J)$-torsion $R$-module, then $M/JM$ is an $I$-torsion module by \cite[Corollary 1.9]{TYY}.
Since by assumption $(JM:_{M}I)=(0:_{M/JM}I)$ is an Artinian module we can use a result due to Melkersson (see \cite[Theorem 1.3]{M}) to get
$M/JM$ is an Artinian module. Moreover, $JM \subseteq (JM:_{M}I)$ is also an Artinian module. By an exact sequence one can conclude $M$ is an Artinian module.
\end{proof}

Next theorem is a generalization of Theorem 2.1 in \cite{RS}.

\begin{thm}\label{artianianess}
Let $J$ be a graded ideal of $R$.
Then the $R$-module
$$
H_{R_{+}, J}^{i}(M)/(\mathfrak{m}_{0}R +J)H_{R_{+}, J}^{i}(M)
$$ is Artinian for all $i\geq 0$.
\end{thm}

\begin{proof}
Since $M$ is a finite $R$-module one may verify $R_{+}^{k} \hspace{0,05cm} \Gamma_{R_{+}, \mathfrak{b}}(M) \subset J\Gamma_{R_{+}, J}(M)$ for some integer $k$. It implies $(\mathfrak{m}_{0} + R_{+})^{k} \Gamma_{R_{+}, J}(M) \subset \mathfrak{m}_{0}\Gamma_{R_{+}, J}(M)+JH_{R_{+}, J}^{n}(M),$ so that one may conclude
$$
\Gamma_{R_{+}, J}(M)/(\mathfrak{m}_{0}R+J)\Gamma_{R_{+}, J}(M)
$$
is Artinian.

Now assume, by induction, that $i>0$ and we have shown
$$
H_{R_{+}, J}^{i-1}(M')/(\mathfrak{m}_{0}R +J)H_{R_{+}, J}^{i-1}(M')
$$
is an Artinian module for any finite graded $R$-module $M'$.
In view of \cite[Corollary 1.13]{TYY}, we can assume $M$ is an $(R_{+}, J)$-torsion free $R$-module.
Since $\Gamma_{R_{+}}(M) \subseteq \Gamma_{R_{+}, J}(M)$, there exists a homogeneous element $x\in R_{+}$ which is $M$-regular. Say, $\text{deg}(x)=a$.
The exact sequence
$$
0 \rightarrow M \stackrel{x}{\rightarrow} M[-a] \rightarrow (M/xM)[-a] \rightarrow 0
$$
induces an exact sequence

$$
\begin{array}{r}
 \cdot \cdot \cdot \rightarrow H_{R_{+},J}^{i-1}(M/xM)\\
 \rightarrow H_{R_{+},J}^{i}(M) \stackrel{x}{\rightarrow} H_{R_{+},J}^{i}(M)[-a] \rightarrow H_{R_{+},J}^{i}(M/xM)[-a] \rightarrow \cdot \cdot \cdot
\end{array}
$$

\noindent Hence, we have the exact sequence
$$
\frac{H_{R_{+},J}^{i-1}(M/xM)}{(\mathfrak{m}_{0}R+J)H_{R_{+}, J}^{i-1}(M/xM)} \rightarrow \frac{H_{R_{+},J}^{i}(M)}{(\mathfrak{m}_{0}R+J)H_{R_{+}, J}^{i}(M)} \stackrel{x}{\rightarrow}
$$
$$
\hspace{-6cm} \rightarrow \frac{xH_{R_{+},J}^{i}(M)[-a]}{x(\mathfrak{m}_{0}R+J)H_{R_{+}, J}^{i}(M))[-a]}   \rightarrow 0.
$$
By induction hypothesis, we have
$
\frac{H_{R_{+},J}^{i-1}(M/xM)}{(\mathfrak{m}_{0}R +J)H_{R_{+}, J}^{i-1}(M/xM)}
$
is Artinian. It shows the kernel of the multiplication by $x$ on $\frac{H_{R_{+},J}^{i}(M)}{(\mathfrak{m}_{0}R +J)H_{R_{+}, J}^{i}(M)}$ is an Artinian $R$-module.
By using \cite[Corollary 1.13]{TYY}, one can deduce
$\frac{H_{R_{+},J}^{i}(M)}{(\mathfrak{m}_{0}R +J)H_{R_{+}, J}^{i}(M)}$ is an $((x),J)$-torsion $R$-module. To conclude the proof, we apply Proposition \ref{prop_melkerson} in order to obtain $\frac{H_{R_{+},J}^{i}(M)}{ (\mathfrak{m}_{0}R+J)H_{R_{+}, J}^{i}(M)}$ is Artinian. This completes the inductive step.

\end{proof}

\begin{thm}\label{anulamento}
Let $J$ be a graded ideal of $R$.
If $\dim M/(\mathfrak{m}_{0}R+J)M=d,$ then
$$
H_{R_{+}, J}^{i}(M)=0,
$$
for all $i> d$.
\end{thm}

\begin{proof}
We argue by induction. Let $n=\dim_{R} M$. Suppose $n=-1$, so there is nothing to be proved, as $M=0$. Assume then $n\geq 0$ and that the result is established for $R$-modules of dimension smaller than $n$.
By \cite[Corollary 2.5]{TYY},
$
H_{R_{+}, J}^{i}(\Gamma_{J}(M)) \cong H_{R_{+}}^{i}(\Gamma_{J}(M))
$
for all $i\geq 0$, once $\Gamma_{J}(M)$ is an $J$-torsion module.
By using the fact that
$\sqrt{\text{ann}\left(\frac{\Gamma_{J}(M)}{\mathfrak{m}_{0}\Gamma_{J}(M)}\right)}=\sqrt{\mathfrak{m}_{0}R+\text{ann}(\Gamma_{J}(M))}$ and
$\sqrt{ \text{ann}\left(\frac{M/JM}{\mathfrak{m}_{0}(M/JM)}\right)} = \sqrt{\mathfrak{m}_{0}R+J+\text{ann}(M)}$, one can conclude
$
\dim \frac{\Gamma_{J}(M)}{\mathfrak{m}_{0}\Gamma_{J}(M)} \leq \dim \frac{M/JM}{\mathfrak{m}_{0}(M/JM)}=d.
$
We use then \cite[Lemma 3.4]{BH} to obtain
$
H_{R_{+}, J}^{i}(\Gamma_{J}{(M)})=0
$
for all $i > d$.
Moreover, the exact sequence $0 \rightarrow \Gamma_{J}(M) \rightarrow M \rightarrow M/\Gamma_{J}(M) \rightarrow 0$ yields the long exact sequence
\begin{equation}\label{long}
\cdot \cdot \cdot \rightarrow H_{R_{+}, J}^{i}(\Gamma_{J}(M)) \rightarrow H_{R_{+}, J}^{i}(M) \rightarrow H_{R_{+}, J}^{i}(M/\Gamma_{J}(M))\rightarrow \cdot \cdot \cdot
\end{equation}
We then derive
$$
H_{R_{+}, J}^{i}(M) \cong H_{R_{+}, J}^{i}(M/\Gamma_{J}(M)),
$$
for all $i > d$. We have
$$
\dim \frac{M/\Gamma_{J}(M)}{(\mathfrak{m}_{0}R+J)M/\Gamma_{J}(M)}= \dim \frac{M}{(\mathfrak{m}_{0}R+J)M+\Gamma_{J}(M)}\leq d.
$$
In this way, we can assume $M$ is an $J$-torsion-free module. So there exists an homogeneous element $a\in J$ which is a non zero-divisor on $M$.
The exact sequence $0 \rightarrow M \stackrel{a}{\rightarrow} M \rightarrow M/aM \rightarrow 0$ yields the long exact sequence
$$
\cdot \cdot \cdot \rightarrow H_{R_{+}, J}^{i}(M) \stackrel{a}{\rightarrow} H_{R_{+}, J}^{i}(M)\rightarrow H_{R_{+}, J}^{i}(M/aM) \rightarrow \cdot \cdot \cdot
$$
Once $\dim_{R} M/aM=n-1$ and $\dim_{R} \frac{M/aM}{(\mathfrak{m}_{0}R+J)M/aM}=d$, we use the inductive hypothesis to conclude
$H_{R_{+}, J}^{i}(M/aM)=0$ for all $i > d$. Then the above long exact sequence says $aH_{R_{+},J}^{i}(M)=H_{R_{+}, J}^{i}(M)$. By Nakayama's Lemma, we obtain the desired result.
\end{proof}

Next result is a similar result to \cite[Lemma 3.4(b)]{BH}.

\begin{thm}\label{Theorem5}
Let $J$ be a graded ideal of $R$.
Then
$$
\dim M/(\mathfrak{m}_{0}R+J)M = {\rm sup} \{i \ :  H^{i}_{R_{+},J}(M)\neq 0  \}.
$$

\end{thm}

\begin{proof}
Consider the exact sequence
$$
0 \rightarrow (\mathfrak{m}_{0}R +J)M \rightarrow M \rightarrow M/(\mathfrak{m}_{0}R+J)M \rightarrow 0,
$$
which yields the long exact sequence
$$
\cdot \cdot \cdot \rightarrow H_{R_{+}, J}^{i}(M) \rightarrow H_{R_{+}, J}^{i}(M/(\mathfrak{m}_{0}R + J)M)\rightarrow H_{R_{+}, J}^{i+1}((\mathfrak{m}_{0}R + J)M) \rightarrow \cdot \cdot \cdot.
$$
Now note that
$$
\dim \frac{(\mathfrak{m}_{0}R+J)M}{(\mathfrak{m}_{0}R+J)(\mathfrak{m}_{0}R+J)M} \leq \dim \frac{M}{(\mathfrak{m}_{0}R+J)^{2}M}=\dim \frac{M}{(\mathfrak{m}_{0}R+J)M}=d;
$$
thus, because of Theorem \ref{anulamento}, we obtain $H_{R_{+},\mathfrak{b}}^{i}((\mathfrak{m}_{0}R+J)M)=0$ for $i> d$.
On the other hand, by using \cite[Corollary 2.5]{TYY} and \cite[Corollary 35.20]{HIO}, one sees
$$
H_{R_{+}, J}^{d} (M/(\mathfrak{m}_{0}R + J)M ) \cong H_{R_{+}}^{d}(M/(\mathfrak{m}_{0}R + J)M)  \cong H_{(R/\mathfrak{m}_{0}R)_{+}}^{d}(M/(\mathfrak{m}_{0}R + J)M).
$$
This last cohomology module is nonzero by \cite[Corollary 36.19]{HIO}. By observing the above long exact sequence we get the desired result.
\end{proof}

\section{Depth $(I,J)$ on graded module}

In this section, we work with the concept of depth of a pair homogeneous ideals of $R$, and obtain an expression for it as $M$ and $R$ are both Cohen-Macaulay. For the ordinary case the \cite[Theorem 6.2.7]{BS} says
$${\rm depth}_{I}(M)= {\rm inf}\{ i \in \N_0 \ : \ H_{I}^{i}(M) \neq 0  \},$$
for an $R$-ideal $I$ and a finite $R$-module $M$ such that $IM\neq M$. For the case of a pair $(I,J)$ ideals it was defined in \cite[Definition 3.1]{AAS}; we then recall this definition as follows.

Throughout this section we assume $R$ and $M$ are as in section 3.

\begin{defi} Let $I,$ $J$ be two homogeneous ideals of the graded ring $R$ and $M$ a graded $R$-module. We define depth of $(I,J)$ on $M$ by
$${\rm depth}(I,J,M)={\rm inf}\{{\rm depth}({\mathfrak a},M) \ | \ {\mathfrak a}\in \tilde{W}(I,J) \}$$
\noindent it this infimum exists, and $\infty$ otherwise.
\end{defi}

\begin{rem}\label{remark8}{\rm
In the general case, by \cite[Theorem 4.1 and Theorem 3.2]{TYY} or \cite[Proposition 3.3]{AAS}, we have

$$
{\rm depth}(I,J,M)={\rm inf}\{i\in\N_0| H_{I,J}^i(M)\neq 0\}.$$}
\end{rem}

Let $I,J,K$ be three homogeneous ideals of $R$. Let us introduce the number
$$
\text{ht}(I,J,K):= \text{inf} \{ \text{ht}(\mathfrak{p}) \ | \ \mathfrak{p}\in W(I,J)\cap V(K) \}.
$$
Note that when $J=(0)$,
$$
\text{ht}(I,J,K)=\text{ht}(I+K).
$$

\begin{prop}
Assume $M$ is a Cohen-Macaulay $R$-module and $R$ is a Cohen-Macaulay ring. Set $I=\sqrt{\emph{ann}_{R}M}$. Then
$$
\emph{depth} \hspace{0.05cm} (R_{+},J,M)=\emph{ht}(R_{+},I,J)-\emph{ht}(I).
$$

\noindent In particular, $\emph{ht}(R_{+},I,J)-\emph{ht}(I)={\rm inf}\{i\in\N_0 \ | \ H_{R_+,J}^i(M)\neq 0\}.$
\end{prop}

\begin{proof}
By \cite[Theorem 4.1]{TYY} (or \cite[Proposition 3.3]{AAS}),
$$
\text{depth} \hspace{0.05cm} (R_{+},J,M)=\text{inf}\{ \text{depth} \hspace{0.05cm} M_{\mathfrak{p}} \ | \ \mathfrak{p}\in W(R_{+},J)  \}.
$$
Note that
$$
\text{depth} \hspace{0.05cm} (R_{+},J,M)=\text{inf}\{ \text{depth} \hspace{0.05cm} M_{\mathfrak{p}} \ | \ \mathfrak{p}\in W(R_{+},J) \cap V(I) \},
$$
once $M_{\mathfrak{p}}=0$ as $\mathfrak{p}\not\in V(I)$, that is, $\text{depth} \hspace{0.05cm} M_{\mathfrak{p}}= \infty$.
Since $M$ is Cohen-Macaulay, in particular,
$$
\text{depth} \hspace{0.05cm} M_{\mathfrak{p}}= \dim M_{\mathfrak{p}}
$$
for each $\mathfrak{p}\in W(R_{+},J)$. But $R$ is also Cohen-Macaulay by hypothesis, so $\dim M_{\mathfrak{p}}=\dim R_{\mathfrak{p}}/I_{\mathfrak{p}}=\dim R_{\mathfrak{p}}- \text{ht} I_{\mathfrak{p}}$. By \cite[Lemma 1.2.2]{L},
all minimal primes of $I$ have the same height, so that, $\text{ht} I_{\mathfrak{p}}=\text{ht} I$ for each $\mathfrak{p}$. By combining the above arguments, we obtain
$$
\text{depth} \hspace{0.05cm} (R_{+},J,M)=\text{inf} \{\dim R_{\mathfrak{p}}- \text{ht} I \ | \ \mathfrak{p}\in W(R_{+},J) \cap V(I)\}.
$$
This last one equals $\text{ht}(R_{+},J,K)- \text{ht} I$, by definition. The proof is completed.
\end{proof}

\begin{prop}
Let $J$ be a graded ideal of $R$. Then any integer $i$ for which $H_{R_+,J}^i(M)\neq 0$ satisfies $${\rm depth}(R_+,J,M)\leq i\leq \dim M/({\mathfrak m_0}R+J)M.$$
\end{prop}
\begin{proof} The proof follows by Remark \ref{remark8} and Theorem \ref{Theorem5}.
\end{proof}

\begin{cor}\label{cor8} There is exactly one integer $i$ for which $H_{R_+,J}^i(M)\neq 0$ if and only if
$$
{\rm depth}(R_+,J,M)=\dim M/({\mathfrak m_0}R+J)M.
$$
\end{cor}

\section{Basic finiteness, vanishing theorem and\\ asymptotical stability}

In this section, we generate two classical results about the graded components of local cohomology for the case of local cohomology with respect to a pair of ideals (Proposition \ref{vanishing} and Theorem \ref{finitely generated components}). Finally,
we obtain a result on the asymptotical stability of the sequence $\{{\rm Ass}_{R_0}(H_{R_+,J}^i(M))_n\}_{n\in \Z}$.

Throughout this section, we will make the following assumptions.

Let $R= \oplus_{d\geq 0}R_{d}$ denote a positively graded commutative Noetherian ring, which is standard, that is, $R=R_{0}[R_{1}]$ and let $M$ be a finitely generated module over $R$.
Assume $R_{0}$ is a local ring of maximal ideal $\mathfrak{m}_{0}$ and residual field $R_0/\mathfrak{m}_0$ is infinite. Set $R_{+}=\oplus_{i>0} R_{i}$, the irrelevant ideal of $R$.

It is well-known that $H_{R_+}^i(M)_{n}$ is finitely generated $R_{0}$-module for all $n \in \mathbb{Z}$ and $H_{R_+}^i(M)_{n}=0,$ for $n$ sufficiently large. Next assertion gives a positive answer for the case of cohomology modules with respect to a pair of ideals.

\begin{rem}\label{basic}
Let $I,K,J$ be arbitrary ideals of $R$, and let $M$ be a $K$-torsion $R$-module. It is easy to check that
$
H_{I+K,J}^{i}(M)\cong H_{I,J}^{i}(M)
$
for all $i\in \N_0$.
\end{rem}

\begin{prop}\label{vanishing}
Let $J$ be an ideal generated by elements of zero degree. Set $\mathfrak{b}=\mathfrak{b}_0 + R_+$. Then for $i \geq 1$, $$H_{\mathfrak{b},J}^i(M)_{n}=0,$$ for $n$ sufficiently large.
\end{prop}

\begin{proof}
Firstly we will prove the assertion as $\mathfrak{b}_0=0$.
We proceed by induction on $\dim M$. If $\dim M=0$ then
by \cite[Theorem 4.7(1)]{TYY}, the result is clearly true in this case. Suppose now $\dim M>0$ and the result is established for finitely generated modules of dimension less than $\dim M$.

By \cite[Corollary 2.5]{TYY},
$
H_{R_{+}, J}^{i}(\Gamma_{J}(M)) \cong H_{R_{+}}^{i}(\Gamma_{J}(M)),
$
for all $i\geq 0$. The exact sequence $0 \rightarrow \Gamma_{J}(M) \rightarrow M \rightarrow M/\Gamma_{J}(M) \rightarrow 0$ yields then the long exact sequence
$$
H_{R_{+}}^{i}(\Gamma_{J}(M))_{n} \rightarrow H_{R_{+}, J}^{i}(M)_{n} \rightarrow H_{R_{+}, J}^{i}(M/\Gamma_{J}(M))_{n}\rightarrow H_{R_{+}}^{i+1}(\Gamma_{J}(M))_{n}.
$$
One can conclude that
$H_{R_{+}, J}^{i}(M)_{n}$ is isomorphic to $H_{R_{+}, J}^{i}(M/\Gamma_{J}(M))_{n}$ for all $n$ sufficiently large. Hence, we can assume that $M$ is an $J$-torsion-free $R$-module. As $J$ is generated by homogeneous elements of zero degree, by a Prime Avoidance Lemma there exists an element $x$ in $J$ of zero degree which is a non zero divisor on $M$.

The exact sequence
$$
0 \rightarrow M \stackrel{x}{\rightarrow} M \rightarrow M/xM \rightarrow 0
$$
induces an exact sequence

$$
 H_{R_{+},J}^{i}(M)_{n} \stackrel{x}{\rightarrow} H_{R_{+},J}^{i}(M)_{n} \rightarrow H_{R_{+},J}^{i}(M/xM)_{n}.
$$
By induction, for every $i\geq 1$, $H_{R_{+},J}^{i}(M/xM)_{n}=0$ for all $n$ sufficiently large;
the above exact sequence yields then an epimorphism
$$
H_{R_{+},J}^{i}(M)_{n} \stackrel{x}{\rightarrow} H_{R_{+},J}^{i}(M)_{n},
$$
for all $n>>0$, so
$$
H_{R_{+},J}^{i}(M)_{n}=xH_{R_{+},J}^{i}(M)_{n}.
$$
Nakayama's Lemma completes the first part.

Now we conclude the proposition. We proceed by induction on $\dim M$. If $\dim M=0$, the result is clearly true.

Consider the exact sequence $0 \rightarrow \Gamma_{\mathfrak{b}_0}(M) \rightarrow M \rightarrow M/\Gamma_{\mathfrak{b}_0}(M) \rightarrow 0$, which yields, by Remark \ref{basic}, the long exact sequence
$$
H_{R_{+}}^{i}(\Gamma_{\mathfrak{b}_0}(M))_{n} \rightarrow H_{R_{+}, J}^{i}(M)_{n} \rightarrow H_{R_{+}, J}^{i}(M/\Gamma_{\mathfrak{b}_0}(M))_{n}\rightarrow H_{R_{+}}^{i+1}(\Gamma_{\mathfrak{b}_0}(M))_{n}.
$$
So, by using the first part, we obtain the isomorphism
$$H_{R_{+}, J}^{i}(M)_{n} \rightarrow H_{R_{+}, J}^{i}(M/\Gamma_{\mathfrak{b}_0}(M))_{n}
$$
for $n$ large. We then may assume there exists an element $x \in \mathfrak{b}_0$ which is a nonzero divisor on $M$. The final result is obtained by following the same arguments as in the last paragraph of the first part.

\end{proof}



\begin{lem}
If $J$ is generated by elements of zero degree and $\Gamma_{R_+,J}(M)_{n}=0$ for $n>>0$, then, for each $i\geq 0$, $H_{R_+,J}^i(M)_{n}$ is a finitely generated $R_{0}$-module for all integer $n$.
\end{lem}

\begin{proof}
We proceed by induction on $i$. By hypothesis and by the fact that $M_{n}$ for $n<<0$, we conclude that there exists $u\in \mathbb{N}$ such that
$R_{+}^{u}H_{R_+,J}^0(M)=0$. Note that
$R_{+}^{i}H_{R_{+},J}^0(M)/R_{+}^{i+1}H_{R_{+},J}^0(M)$ is a Noetherian $R/R_{+}$-module, that is, a Noetherian $R_{0}$-module for each $i=0,...,u-1$. One then obtains
$H_{R_+,J}^0(M)$ is a finitely generated $R_{0}$-module, and so is $H_{R_+,J}^0(M)_n$ for all $n \in \mathbb{Z}$.

Suppose now $i>0$ and the result is established for smaller values of $i$. Because of the graded isomorphism
$H_{R_+,J}^i(M) \cong H_{R_+,J}^i(M/\Gamma_{R_+,J}(M))$ for $i>0$ (see \cite[Corollary 1.13]{TYY}), we can assume
$M$ is an $(R_+,J)$-torsion-free $R$-module (and so is an $R_+$-torsion-free module). Then there exists an element $x\in R_+$ which is a non-zero divisor on $M$, say $\text{deg}(x)=a$. The exact sequence
$$
0 \rightarrow M \stackrel{x}{\rightarrow} M[a] \rightarrow (M/xM)[a] \rightarrow 0
$$
induces an exact sequence

$$
\begin{array}{r}
 H_{R_{+},J}^{i-1}(M)_{n+a} \rightarrow H_{R_{+},J}^{i-1}(M/xM)_{n+a}
 \rightarrow H_{R_{+},J}^{i}(M)_n \stackrel{x}{\rightarrow} H_{R_{+},J}^{i}(M)_{n+a}. \end{array}
$$
From the above exact sequence and Proposition \ref{vanishing} one can deduce \\
$\Gamma_{R_+,J}(M/xM)_n=0$ for $n$ sufficiently large.
From the inductive hypothesis it then follows that
$H_{R_{+},J}^{i-1}(M/xM)_{q}$ is finitely generated for all $q \in \mathbb{Z}$. Again by Proposition \ref{vanishing}, there exists $s \in \mathbb{Z}$ such that $H_{R_{+},J}^{i-1}(M/xM)_{n}=0$ for all $n\geq s$ and $H_{R_{+},J}^{i}(M)_{n}=0$ for all $n\geq s-a$.

Fix $n\in \mathbb{Z}$ and let $k\geq 0$ be an integer such that $n+ka\geq s-a$. So
$H_{R_{+},J}^{i}(M)_{n+ka}=0$. For each $j=0,..,k-1$, we have the exact sequence
$$
H_{R_{+},J}^{i-1}(M/xM)_{n+(j+1)a}
 \rightarrow H_{R_{+},J}^{i}(M)_{n+ja} \stackrel{x}{\rightarrow} H_{R_{+},J}^{i}(M)_{n+(j+1)a}.
$$
In conclusion, we obtain
$
H_{R_{+},J}^{i}(M)_{n+ja}
$
is a finitely generated for $j=k-1,k-2,...,1,0$, so that
$
H_{R_{+},J}^{i}(M)_{n}
$
is a finitely generated for all $n \in \mathbb{Z}$.

\end{proof}

\begin{thm} \label{finitely generated components}
If $J$ is generated by elements of zero degree, then, for all $i\geq 1$, $H_{R_+,J}^i(M)_{n}$ is a finitely generated $R_{0}$-module for all $n\in \mathbb{Z}$.
\end{thm}

\begin{proof}
It follows from the graded isomorphism $H_{R_+,J}^i(M) \cong H_{R_+,J}^i(M/\Gamma_{R_+,J}(M))$ for $i>0$ and the above lemma.
\end{proof}

\vspace{0.5cm}

Consider the following definition. We say ${\rm Ass}_{R_0}(H_{R_+,J}^i(M)_n)$ is \emph{asymptotically increasing for} $n\to -\infty$, if there exists an $n_0\in \Z$ such that $${\rm Ass}_{R_0}(H_{R_+,J}^i(M)_n)\subseteq{\rm Ass}_{R_0}(H_{R_+,J}^i(M)_{n+1})$$ for all $n\leq n_0.$

\begin{lem}\label{stable} Let $M$ be a finitely generated graded $R$-module and $J$ an arbitrary ideal. Let $i\in \N_0$ be such that $H_{R_+,J}^j(M)_n$ is finitely generated $R_0$-module for all $j<i$ and $n<<0$. Then ${\rm Ass}_{R_0}(H_{R_+,J}^i(M)_n)$ is asymptotically increasing for $n\to -\infty$.
\end{lem}

\begin{proof} We prove by induction on $i$. The case $i=0$ is trivial as $H_{R_+,J}^0(M)_n=0$ for all $n\ll 0$. So, let $i>0$. In view of the natural graded isomorphism $H_{R_+,J}^k(M)\cong H_{R_+,J}^k(M/\Gamma_{R_+,J}(M))$ for all $k\geq 1$ (see \cite[Corollary 1.13 (4)]{TYY}), we may assume $M$ is an $(R_+,J)$-torsion-free $R$-module. Since
$\Gamma_{R_{+}}(M) \subseteq \Gamma_{R_{+},J}(M)$,
we now use \cite[Lemma 2.1.1]{BS} and \cite[Proposition 1.5.12]{BH} to deduce that $R_1$ contains an element $x$ which is a non zero-divisor on $M$. The exact sequence
$0 \rightarrow M(-1) \stackrel{x}{\rightarrow} M \rightarrow M/xM \rightarrow 0$ induces a graded long exact sequence

$$H_{R_+,J}^{k-1}(M) \rightarrow H_{R_+,J}^{k-1}(M/xM) \rightarrow H_{R_+,J}^k(M)(-1) \rightarrow H_{R_+,J}^k(M).$$

\noindent Through this sequence, we obtain
$
H_{R_+,J}^{j-1}(M/xM)
$
is finitely generated for all $j<i$. So, by induction, there exists some $n_1\in \Z$ such that
$$
{\rm Ass}_{R_0}(H_{R_+,J}^{j-1}(M/xM)_n)={\rm Ass}_{R_0}(H_{R_+,J}^{j-1}(M/xM)_{n_1})=:{\mathcal A}\,\,\, { \rm for \ all }\,\, n \leq n_1.
$$
\noindent  Moreover, there exists some $n_2<n_1$ such that $H_{R_+,J}^{i-1}(M)_{n+1}=0$ for all $ n\leq n_2$.  So, for each $n\leq n_2$, we have an exact sequence of $R_0$-modules

$$0 \rightarrow H_{R_+,J}^{i-1}(M/xM)_{n+1} \rightarrow H_{R_+,J}^i(M)_n \rightarrow H_{R_+,J}^i(M)_{n+1},$$
\noindent induced by the above exact sequence. This shows that
$$
{\mathcal A}\subseteq {\rm Ass}_{R_0}(H_{R_+,J}^i(M)_n)\subseteq {\mathcal A}\cup  {\rm Ass}_{R_0}(H_{R_+,J}^i(M)_{n+1}) , \ \text{for all}  \ n \leq n_2.
$$
One can then conclude that
$$
{\rm Ass}_{R_0}(H_{R_+,J}^i(M)_n)\subseteq {\rm Ass}_{R_0}(H_{R_+,J}^i(M)_{n+1}) , \ \text{for all} \ n < n_2.
$$
\end{proof}

\begin{thm} \label{ass_stable} Let $M$ be a finitely generated graded $R$-module, $J$ an ideal of $R$ generated by elements of zero degree, and let $i\in \N$ be such that $H_{R_+,J}^j(M)_n$ is finitely generated $R_0$-module for all $j<i$ and $n<<0$. Then
there exists a finite subset $X$ of $\emph{Spec}(R_0)$ such that ${\rm Ass}_{R_0}(H_{R_+,J}^i(M)_n)=X$ for $n<<0$.
\end{thm}

\begin{proof}
This is a consequence of Lemma \ref{stable} and Theorem \ref{finitely generated components}, once \\ ${\rm Ass}_{R_0}(H_{R_+,J}^i(M)_n)$ is finite for all integer $n$.
\end{proof}

If the sequence $\{{\rm Ass}_{R_0}(H_{R_+,J}^i(M))_n\}_{n\in \Z}$ satisfies the assertion in Theorem \ref{ass_stable}, we say it is \emph{asymptotically stable for} $n \to -\infty$.

\vspace{0.2cm}

 \noindent As an immediate consequence of Corollary \ref{cor8}, Lemma \ref{stable} and Theorem \ref{ass_stable} we get the following result.

 \begin{cor} Let $M$ be a finitely generated graded $R$-module.
\begin{enumerate}
\item[\emph{(i)}] If ${\rm depth}(R_+,J,M)=\dim M/({\mathfrak m_0}R+J)M$ then, for all $i\geq 0$, \\ $\{{\rm Ass}_{R_0}(H_{R_+,J}^i(M))_n\}_{n\in \Z}$ is asymptotically stable for $n\to -\infty$;

\item[\emph{(ii)}] If $t={\rm depth}(R_+,J,M)$, then $\{{\rm Ass}_{R_0}(H_{R_+,J}^t(M))_n\}_{n\in \Z}$ is asymptotically stable for $n\to -\infty$.
\end{enumerate}
\end{cor}

Last Corollary, item $\text{(ii)}$, was showed in \cite[Proposition 5.6]{BH} for the case $J=0.$

\textbf{Department of Mathematics, Institute of Mathematics and Computer Science, ICMC, University of S\~{a}o Paulo, BRAZIL.
}

\emph{E-mail address}: apoliano27@gmail.com

\hspace{1cm}

\textbf{Department of Mathematics, Institute of Mathematics and Computer Science, ICMC, University of S\~{a}o Paulo, BRAZIL.
}

\emph{E-mail address}: vhjperez@icmc.usp.br

\hspace{1cm}
\end{document}